\documentclass[11pt]{amsart}
\usepackage{a4wide,amssymb}
\usepackage[all]{xy}
\newcommand{\IN}{\mathbb N}

\newcommand{\IR}{\mathbb R}
\newcommand{\Ra}{\Rightarrow}
\newcommand{\w}{\omega}
\newcommand{\E}{\mathcal E}
\newcommand{\U}{\mathcal U}
\newcommand{\I}{\mathcal I}
\newcommand{\e}{\varepsilon}
\newcommand{\diam}{\mathrm{diam}\,}
\newcommand{\add}{\mathrm{add}}
\newcommand{\cov}{\mathrm{cov}}
\newcommand{\cof}{\mathrm{cof}}
\newcommand{\non}{\mathrm{non}}
\newcommand{\cf}{\mathrm{cf}}

\newtheorem{theorem}{Theorem}[section]

\newtheorem{proposition}[theorem]{Proposition}
\newtheorem{corollary}[theorem]{Corollary}
\newtheorem{problem}[theorem]{Problem}
\newtheorem{question}[theorem]{Question}
\newtheorem{claim}[theorem]{Claim}
\theoremstyle{definition}
\newtheorem{definition}[theorem]{Definition}
\newtheorem{example}[theorem]{Example}
\newtheorem{remark}[theorem]{Remark}

\begin{document}


\title{Functional boundedness of balleans}
\author{Taras Banakh, Igor Protasov}
\address{T.Banakh: Faculty of Mechanics and Mathematics, Ivan Franko National University of Lviv (Ukraine) and\\ Institute of Mathematics, Jan Kochanowski University in Kielce (Poland)}
\email{t.o.banakh@gmail.com}
\address{I.Protasov: Faculty of Cybernetics, Taras Shevchenko National University in Kyiv, Ukraine}
\email{i.v.protasov@gmail.com}
\begin{abstract} We survey some results and pose some open problems related to boundedness of real-valued functions on balleans and coarse spaces. A special attention is payed to balleans on groups.
\end{abstract}
\maketitle

\section{Introduction}

A {\em ballean} is a pair $(X,\mathcal E_X)$ consisting of a set $X$ and a family $\mathcal E_X$ of subsets of the square $X\times X$  satisfying the following three axioms:
\begin{enumerate}
\item each $E\in\mathcal E_X$ contains the diagonal $\Delta_X=\{(x,x):x\in X\}$ of $X$;
\item for any $E,F\in\mathcal E_X$ there exists $D\in\mathcal E_X$ such that $E\circ F^{-1}\subset D$, where \newline $E\circ F:=\{(x,z):\exists y\;\;(x,y)\in E,\;(y,z)\in F\}$ and $F^{-1}:=\{(y,x): (x,y)\in F\}$.
\item $\bigcup\E_X=X\times X$.
\end{enumerate}
The family $\mathcal E_X$ is called the {\em ball structure} of the ballean $(X,\mathcal E_X)$ and its elements are called {\em entourages}. For each entourage $E\in\mathcal E_X$ and point $x\in X$ we can consider the set $E[x]:=\{y\in X:(x,y)\in E\}$, called the {\em ball of radius $E$} centered at $x$.
For a subset $A\subset X$ the set $E[A]:=\bigcup_{a\in A}E[x]$ is called the {\em $E$-neighborhood} of  $A$. Observe that $E=\bigcup_{x\in X}\{x\}\times E[x]$, so the entourage $E$ can be recovered from the family of balls $E[x]$, $x\in X$.


For a ballean $(X,\E_X)$ and a subset $Y\subset X$ the ballean $(Y,\E_X{\restriction}Y)$ endowed with the ball structure
$$\E_X{\restriction}Y:=\{(Y\times Y)\cap E:E\in\E_X\}$$is called a {\em subballean} of $X$.

Any metric space $(X,d)$ carries a natural ball structure $\{E_\e:\e>0\}$ consisting of the entourages $E_\e:=\{(x,y)\in X\times X:d(x,y)<\e\}$.

A ballean $(X,\mathcal E)$ is called a {\em coarse space} if for any entourage $E\in\mathcal E_X$, any subset $F\subset E$ with $\Delta_X\subset F$ belongs to $\mathcal E_X$.  In this case $\mathcal E_X$ is called the {\em coarse structure} of $X$. For a coarse structure $\E$, a subfamily  $\mathcal B\subset\E$ is called a {\em base} of $\mathcal E$ if each set $E\in\E$ is contained in some set $B\in \mathcal B$. It is easy to see that each base of a coarse structure is a ball structure. On the  other hand, each ball structure $\mathcal E$ on a set $X$ is a
base of the unique coarse structure $${\downarrow}\mathcal E:=\{E\subset X\times X:\Delta_X\subset E\subset F\mbox{ for some }F\in\mathcal B\}.$$

If the ball (or coarse) structure $\E_X$ is clear from the context, we shall write $X$ instead of $(X,\E_X)$. More information on balleans and coarse spaces can be found in
\cite{CH}, \cite{Dr}, \cite{DH},  \cite{NYu}, \cite{PB}, \cite{PZ}, \cite{Roe}.

A subset $B\subset X$ of a ballean $(X,\E_X)$ is called {\em bounded} if $B\subset E[x]$ for some $E\in\E_X$ and $x\in X$. A ballean $X$ is {\em bounded} if $X$ is a bounded set in $(X,\E_X)$.

The family $\mathcal B_X$ of all bounded subsets is called the {\em bornology} of the ballean $(X,\E_X)$. If the ballean $X$ is unbounded, then the bornology $\mathcal B_X$ is an ideal of subsets of $X$. A family $\I$ of subsets of a set $X$ is called an {\em ideal} on $X$ if $\I$ is closed under finite unions and taking subsets, and $\bigcup\I=X\notin\I$.

Given a ballean $X$ a real-valued function  $f: X \to\mathbb{R}$ is called
\begin{itemize}
\item {\it bornologous} if $f(B)$ is bounded for each $B\in\mathcal{B}_{X}$;
\item {\it macro-uniform} if  for every $E\in\mathcal{E}_{X}$ the supremum $\sup_{x\in X}\diam  f(E[x])$ is finite;
\item {\it eventually macro-uniform} if  for every $E\in\mathcal{E}_{X}$ there exist a bounded set $B\in\mathcal{B}_{X}$ such that $\sup_{x\in X\setminus B}\diam  f(E[x])$ is finite;
\item  {\it boundedly oscillating} if there exists $C > 0$  such that, for every $E\in\mathcal{E}_{X}$,
 there is $B\in\mathcal{B}_{X}$   such that $\diam  f(E[x])< C$  for each $x\in X \setminus B$;
\item  {\it slowly oscillating} if for every $E\in\mathcal{E}_{X}$  and $\varepsilon > 0$,
 there exists $B\in\mathcal{B}_{X}$ such that $\diam f(E[x])< \varepsilon$ for each $x\in X \setminus B$;
 \item {\em constant at infinity} if there exists $c\in\mathbb R$ such that for any neighborhood $O_c\subset\IR$ of $c$  the complement $X\setminus f^{-1}(O_c)$ is bounded in $X$.
 \end{itemize}
 \vspace{3 mm}

We say that  a ballean $X$  is

\begin{itemize}
\item  {\it b-bounded} if each bornologous function on $X$  is bounded;
\item {\it mu-bounded} if each macro-uniform function on $X$ is bounded;
\item  {\it emu-bounded} if for every eventually macro-uniform function on $X$ there exists $B\in\mathcal{B}_{X}$  such that $f$ is bounded on $X\setminus B$;
\item  {\it bo-bounded} if for every  boundedly oscillating function $f$ on $X$ there exists $B\in\mathcal{B}_{X}$ such that $f$ is bounded on $X\setminus B$;
 \item  {\it so-bounded} (or {\it pseudobounded}) if for every slowly oscillating function $f$ on $X$ there exists $B\in\mathcal{B}_{X}$ such that $f$ is bounded on $X\setminus B$.
\end{itemize}
 \vspace{3 mm}
 
For any ballean $X$ we have the implications: 
$$
\xymatrix{ 
\mbox{b-bounded}\ar@{=>}[r]&\mbox{mu-bounded}&\mbox{emu-bounded}\ar@{=>}[l]\ar@{=>}[r]&\mbox{bo-bounded}\ar@{=>}[r]&\mbox{so-bounded}.\\
}
$$

If a ballean $X$ is unbounded and $\mathcal{E}_{X}$ has a linearly ordered base, then $X$  is $b$-unbounded iff $X$  is mu-unbounded  iff  $\mathcal{E}_{X}$
has a countable base.

\vspace{3 mm}


We say that a ballean $X$  is {\it locally finite} if $\mathcal{B}_{X}=[X] ^{< \omega}$,
equivalently, each ball in $X$  is finite.

A locally finite ballean $X$ is emu-bounded if and only if $X$  is mu-bounded.
 If a locally finite ballean $X$ is bo-bounded (so-bounded), then every bo-function (so-function) on $X$ is bounded.
\vskip 5pt

\begin{example} 
Let $\kappa,  \mu$  be infinite cardinals  such that $\kappa\leq\mu$.
We denote by $S_{\kappa}$
 the group of all permutations of $\kappa$ and,  for each
  $F\in [S_{\kappa}]^{<\mu} $, put
  $E_{F}= \{(x, y)\in \kappa \times\kappa: y\in \{x\}\cup \{f(x)\}_{f\in F}\}$.
Then the family
$\{E_{F}:  F\in [S_{\kappa}]^{<\mu}\}$
forms a base for some coarse structure  $\mathcal{E}$  on $\kappa$.
We denote the ballean  $(\kappa, \mathcal{E})$ by
$B_{\kappa,\mu}$
and observe that a subset $B\subset \kappa$ is bounded if and only if  $|B|< \mu$.

We show that $B_{\kappa,\mu}$  is emu-bounded (and hence mu-bounded).

First we check that $B_{\kappa,\mu}$  is mu-bounded. Given any unbounded  function $f: \kappa\to\mathbb{R} $, we shall prove that $f$  is not macro-uniform.
We choose a sequence  $(x _{n}) _{n\in\omega}$ in $\kappa$ such that $|f(x_{n+1}) - f(x_{n})|> n$ and define a bijection  $h: \kappa\to \kappa$ by the rule $h(x_{2n})= x_{2n+1} $,  $h(x_{2n+1})= x_{2n} $   and $h(x)=x$    for all    $x\in X\setminus\{x_{n}: n\in\omega\}$.
Then  $\diam  f(E_{\{h\}} [x_{n}])> n$, which means that $f$ is not macro-unifiorm.

Next, we show that the ballean $B_{\kappa,\mu}$ is emu-bounded.
Given an eventually macro-uniform function $f:\kappa\to\IR$, we should find a set $B\subset\kappa$ of cardinality $|B|<\mu$ such that the set $f(\kappa\setminus B)$ is bounded.  If $\mu=\w$, then the eventually macro-uniform function $f$ is macro-uniform and $f$ is bounded by the mu-boundedness of $B_{\kappa,\mu}$. So, we assume that the cardinal $\mu$ is uncountable.

To derive a contradiction, assume that for any subset $B\subset \kappa$ of cardinality $|B|<\mu$, the set $f(\kappa\setminus B)$ is unbounded. By transfinite induction, for every ordinal $\alpha<\mu$ we can construct a sequence $\{x_{\alpha,n}\}_{n\in\w}\subset\kappa$ such that $|f(x_{\alpha,n+1})-f(x_{\alpha,n})|>n$ and $x_{\alpha,n}\notin \{x_{\beta,m}:\beta<\alpha,\;m\in\w\}$ for every $n\in\w$.
Consider a permutation $h\in S_\kappa$ such that $h(x_{\alpha,2n})=x_{\alpha,2n+1}$, $h(x_{\alpha,2n+1})=x_{\alpha,2n}$ for all $\alpha<\mu$ and $n\in\w$. Since the sets $\{x_{\alpha,n}\}_{n\in\w}$, $\alpha<\mu$, are pairwise disjoint,  for any set $B\subset \kappa$ of cardinality $|B|<\mu$ there exists an ordinal $\alpha<\mu$ such that $B\cap\{x_{\alpha,n}\}_{n\in\w}=\emptyset$. Then $\diam f(E_{\{h\}}[x_{\alpha,2n}])\ge \diam \{f(x_{\alpha,2n+1}),f(x_{\alpha,2n})\}>2n$, which means that $f$ is not eventually macro-uniform. But this contradicts the choice of $f$. This contradiction implies that for some set $B\subset\kappa$ of cardinality $|B|<\mu$ the set $f(\kappa\setminus B)$ is bounded.\hfill $\square$
\end{example}

Let $(X, \mathcal{E}),$   $(X^{\prime}, \mathcal{E}^{\prime})$  be balleans.
 A mapping $f:X\to  X^{\prime}$  is  called
 {\it macro-uniform} if for  every
 $E\in \mathcal{E}$ there is
 $E^{\prime}\in \mathcal{E}^{\prime}$
  such that 
  $f(E[x])\subseteq E^{\prime}(f[x])$ for all $x\in X$.
A bijection $f:X\to  X^{\prime}$
is called an {\it asymorphism} if  $f, f ^{-1}$  are macro-uniform.
The balleans $(X, \mathcal{E})$,  $(X ^{\prime}, \mathcal{E} ^{\prime})$
are called {\it coarsely equivalent}  if  there exist large subset
$Y\subseteq X$, $Y^{\prime}\subseteq X$
 such that the balleans 
 $(Y, \mathcal{E}{\restriction}Y)$ and $(Y^{\prime}, \mathcal{E}^{\prime}{\restriction}{Y^{\prime}})$
 are asymorphic.
A subset $Y$ is  called {\it large} if $X=E[Y]$  for some $E\in \mathcal{E}$.

It can be shown that all above defined boundedness are stable under coarse equivalences.


\section{Tall and antitall bornologies}

By a {\em bornology} on a set $X$ we understand any family $\mathcal B$ of subsets of $X$ such that $\bigcup\mathcal B=X$ and $\mathcal B$ is closed under finite unions and taking subsets. In particular, for any infinite cardinal $\kappa$ the family $$[X]^{<\kappa}:=\{A\subset X:|A|<\kappa\}$$ is a bornology on $X$.

A bornology $\mathcal{B}$  on a set $X$  is called {\it tall} if every  infinite subset of $X$  has an infinite bounded subset. A bornology  $\mathcal{B}$ is tall iff every bornologous function on $X$ is bounded. Hence, a ballean $X$ is $b$-bounded iff its bornology $\mathcal{B}_{X}$  is tall.

A bornology $\mathcal{B}$ on $X$ is called {\it antitall} if every unbounded subset of $X$
contains an unbounded subset $Y$
 such that all bounded subset of $Y$ are finite.
 A bornology $\mathcal{B}$ is antitall   iff for every unbounded subset $Y$ of $X$
 there is a bornologous function on $X$ such that $f$ is unbounded on $Y$.

By Proposition 1 \cite{Pr18}, each bornology is the intersection of a tall and an antitall bornologies.

We endow $X$  with the discrete  topology and put $X^{\ast}= \beta X\setminus X$ and
$$\mathcal{B}_{X}^{\sharp} = \{ p\in X^{\ast}: \forall P\in p,\;\;P\notin \mathcal{B}_{X}\}.$$
Then $\mathcal{B}$ is tall iff $\mathcal{B}_{X}^{\sharp} $ is nowhere dense in
$X^{\ast}$, and $\mathcal{B}$ is antitall iff $int \ \mathcal{B}_{X}^{\sharp} $ is dense $\mathcal{B}_{X}^{\sharp} $.

\section{Discrete balleans}

A ballean $X$ is {\em discrete} if $X$ is not bounded and for any $E\in\E_X$ there exists a bounded set $B\subset X$ such that $E[x]=\{x\}$ for all $x\in X\setminus B$.
More information on discrete balleans can be found in  \cite[Chapter 3]{PZ}.

\begin{example}\label{ex:d} Let $X$ be a discrete ballean whose bornology coincides with the family $[X]^{\le\w}$ of at most countable sets on an uncountable set $X$. The ballean $X$ is mu-bounded but each function $f:X\to\mathbb R$ is slowly oscillating and hence eventually macro-uniform. So, $X$ is not emu-bounded and not so-bounded.
\end{example}

Since each macro-uniform function is eventually macro-uniform, each emu-bounded ballean is mu-bounded. Example~\ref{ex:d} shows that the converse implication does not hold even for discrete balleans.  

\begin{proposition} For a ballean $X$,  the following statements are  equivalent
\begin{enumerate}
\item every    function  $f: X\to \mathbb{R}$  is  macro-uniform;
\item  $X$  is  discrete and  $\mathcal{B}_{X} = [X] ^{< \omega}$.
\end{enumerate}
\end{proposition}

\begin{proof} $(2)\Ra  (1)$ is  evident. To verify $(1)\Ra (2)$,  assume that each rael-valued function on $X$ is macro-uniform. In this case $\mathcal{B}_{X} = [X] ^{< \omega}$. To derive a contradiction, assume that $X$  is not discrete. Then there exist  $E \in \mathcal{E}_{X}$  and a sequence $(x_{n})  _{n \in \omega}$  in $X$ such that the subsets $\{E [x_{n}] :  n \in \omega\}$ are  pairwise disjoint and $|E [x_{n}]| > 1$ for all $n \in \omega$. Then the function $f: X\to \mathbb{R}$  defined by $f(x_{n})=n$    and  $f(x)=0$  for $x\notin\{x_{n} :  n \in \omega \}$ is  not macro-uniform, which contradicts our assumption.
\end{proof}

The following characterization can be easily derived from the definitions.

\begin{proposition} For a discrete ballean $X$ the following conditions are equivalent:
\begin{enumerate}
\item $X$ is b-bounded;
\item $X$ is mu-bounded;
\item the bornology of $X$ is tall.
\end{enumerate}
\end{proposition}

A characterization of emu-bounded discrete balleans is more complicated and involves countably complete ultrafilters.

A filter $\mathcal F$ on a set $X$ is called {\em countably-complete} if for any countable subfamily $\mathcal C\subset\mathcal F$ the intersection $\bigcap\mathcal C$ belongs to the filter $\mathcal F$. 

A cardinal $\kappa$ is {\em measurable} if there exists a countably complete free ultrafilter $\U$ on $\kappa$. A measurable cardinal, if exists, is not smaller than the first strongly inaccessible cardinal, see \cite[Ch.10]{Jech}. 
This implies that the existence of measurable cardinals cannot be proved in ZFC. 

\begin{proposition}\label{p3.2}
For a ballean  $X$,  the  following  statements are equivalent:
\begin{enumerate}
\item every  function  $f: X\to  \mathbb{R}$  is  slowly  oscillating;
\item every  function  $f: X\to  \{0,1\}$  is  slowly  oscillating.
\item  $X$ is discrete.
\end{enumerate}
If any unbounded subset of $X$ contains an unbounded set of non-measurable cardinality, then the conditions \textup{(1)--(3)} are equivalent to:
\begin{enumerate}
\item[(4)] every  function  $f: X\to  \mathbb{R}$  is boundedly oscillating;
\item[(5)]  every    function  $f: X\to \mathbb{R}$  is  eventually macro-uniform.
\end{enumerate}
In particular, the conditions \textup{(1)--(5)} are equivalent if no measurable cardinal exists.
\end{proposition}

\begin{proof} The equivalence of the conditions (1)--(3) has been proved in  Theorem 3.3.1 of \cite{PZ} and the implications $(1)\Ra(4)\Ra(5)$ are trivial. Now assuming that any unbounded subset of $X$ contains an unbounded set of non-measurable cardinality, we shall prove that $(5)\Ra(3)$ (which is equivalent to $\neg(3)\Ra\neg(5)$).  Assume that $X$ is not discrete. Then there exists an entourage $E\in \mathcal{E}_{ X}$  such that the set $L=\{x\in X:E[x]\ne\{x\}\}$ is unbounded. Using Zorn's Lemma, find a maximal subset $M\subset L$ such that $E[x]\cap E[y]=\emptyset$ for any distinct elements $x,y\in M$. The maximality of $M$ ensures that $L\subset E^{-1}[E[M]]$, which implies that the set $M$ is unbounded. By our assumption, the unbounded set $M$ contains an unbounded subset $N\subset M$ of non-measurable cardinality $|N|$. Using the Zorn's Lemma, extend the free filter $\{N\setminus B:B\in\mathcal B_X\}$ to any (free) ultrafilter $\mathcal U$. Since $N\in\mathcal U$ and the cardinal $|N|$ is not measurable, the ultrafilter $\mathcal U$ is not countably-complete, so there exists a decreasing sequence of subsets $\{U_n\}_{n\in\w}\subset\mathcal U$ of $N$ such that $\bigcap_{n\in\w}U_n\notin\mathcal U$.
Consider the function $f:X\to \IR$ assiging to each $x\in X$ the number
$$
f(x)=\begin{cases}
n&\mbox{if $x\in U_n\setminus U_{n+1}$ for some $n\in\w$};\\
0&\mbox{otherwise}.
\end{cases}
$$
Assuming that $f$ is eventually macro-uniform, we could find a bounded set $B\subset X$ such that $C:=\sup_{x\in X\setminus B}f(E[x])<\infty$. Choose any number $n\in\w$ with $n>C$. It follows from $U_n\setminus B=U_n\cap(X\setminus B)\in\U$ and $\bigcap_{m\in\w}U_m\notin\mathcal U$ that $U_n\setminus B\not\subset\bigcap_{m\in\w}U_m$ and hence $U_n\setminus B\not\subset U_m$ for some $m\in\w$. So, we can find a point $x\in U_n\setminus B$ such that $x\notin U_m$. It follows that $x\in U_k\setminus U_{k+1}$ for some $k\in[n,m)$ and hence $\diam f(E[x])\ge \diam\{f(x),0\}=k\ge n>C$, which contradicts the definition of $C$. This contradiction shows that the function $f$ is not eventually macro-uniform, which means that the condition (5) does not hold.
\end{proof}

\begin{example}\label{e3.4} Let $X$  be a set of measurable cardinality $|X|$.
We define a coarse  structure  $\mathcal{E}$  on $X$  such that $(X, \mathcal{E})$  is not discrete but each function 
 $f: X \to \mathbb{R}$  is eventually macro-uniform. Write $X$ as the union $X= Y \cup Z$ of two disjoint sets of the same cardinality and fix any bijection  $h: Y \to Z$. Let $\Gamma:=\{(x,h(x)):x\in Y\}$ be the graph of the bijection $h$.
 
 By the measurability of the cardinal $|X|=|Y|$, there exists a countably complete ultrafilter $p$ on $Y$.  For any set $P\in p$ consider the subset 
 $$B_P:=(Y\setminus P)\cup h(Y\setminus P)$$of $X$ and the entourage
 $$E_P:=(B_P\times B_P)\cup \Delta_X\cup \Gamma\cup \Gamma^{-1}$$on $X$. Let $\mathcal E$ be the coarse structure on $X$, generated by the base $\{E_P:P\in p\}$. It is easy to check that the bornology of the coarse space $(X,\mathcal E)$ is generated by the base $\{B_P:P\in p\}$. Using this fact and looking at the definition of the basic entourages $E_{P}$, we can conclude that the ballean $(X,\mathcal E_X)$ is not discrete. 
 
Now we check that each function $f:X\to\IR$ is eventually macro-uniform. Given any set $P\in p$, we should find a bounded set $B\in\mathcal B_X$ such that $\sup_{x\in X\setminus B}\diam f(E_P[x])<\infty$. Observe that for any $x\in P=Y\setminus B_P$ we get $E_P[x]=\{x,h(x)\}=E_P[h(x)]$. This implies that $\diam f(E_P[x])<\infty$ and hence $P=\bigcup_{n\in\w}P_n$ where $P_n=\{x\in P:\diam f(E_P[x])\le n\}$ for $n\in\w$. Since the ultrafilter $p$ is countably complete, there exists $n\in\w$ such that $P_n\in p$. Then for the bounded set $B=B_{P_n}$ and any $x\in X\setminus B$ we get $\diam f(E_{P}[x])\le n$. Indeed, if $x\in Y$, then $x\in Y\setminus B_{P_n}=P_n$ and hence $\diam f(E_{P}[x])=\diam\{x,h(x)\}\le n$ by the definition of $P_n$. If $x\in Z$, then $x\in Z\setminus B_{P_n}=h(P_n)$ and then $h^{-1}(x)\in P_n$ and $E_P[x]=\{x,h^{-1}(x)\}=E_P[h^{-1}(x)]$ and $\diam f(E_P[x])=\diam f(E_P[h^{-1}(x)])\le n$.\hfill$\square$
\end{example}

Given a ballean $X$,  we denote by $X^{\sharp}$   the set of all  ultrafilters $p$  on $X$ such that each member of $p$ is unbounded in $X$. It is easy to see that the set $X^\sharp$ is closed in the Stone-\v Cech extension $\beta X$ of $X$ endowed with the discrete topology.

\begin{proposition}\label{p3.5} For a discrete ballean $X$,  the  following  statements are  equivalent
\begin{enumerate}
\item $X$  is emu-bounded;
\item $X$  is bo-bounded;
\item  $X$  is so-bounded;
\item $X^\sharp$  is  finite and  each ultrafilter $p\in X^{\sharp}$ is countably complete;
\item each ultrafilter $p\in X^{\sharp}$ is countably complete. 
\end{enumerate}
\end{proposition}

\begin{proof} The implication $(3)\Ra (4)$ is proved in Theorem 3.3.2 of \cite{PZ} and $(1)\Ra (2)\Ra(3)$  are evident. 
\smallskip

$(4)\Ra(1)$ Assume that $X^\sharp$ is finite and each ultrafilter $p\in X^\sharp$ is countably complete. To show that $X$ is emu-bounded, take any eventually macro-uniform function $f:X\to\IR$. Let $X^\sharp=\{p_1,\dots,p_k\}$ and observe that $\bigcap_{i=1}^kp_k=\{X\setminus B:B\in\mathcal B_X\}$ where $\mathcal B_X$ is the bornology of the ballean $X$. For every $n\in\w$ consider the set $X_n=\{x\in X:|f(x)|\le n\}$ and observe that $X=\bigcup_{n\in\w}X_n$. For every $i\le k$, the countable completeness of the ultrafilters $p_i\ni \bigcup_{n\in\w}X_n$ yields a number $n_i\in\w$ such that  $X_{n_i}\in p_i$. Then for the number $n=\max_{i\le k}n_i$, we obtain $X_n\in \bigcap_{i=1}^kp_i=\{X\setminus B:B\in\mathcal B_X\}$. Therefore, the complement $B:=X\setminus X_n$ is bounded and $f(X\setminus B)=f(X_n)\subset[-n,n]$ is bounded in the real line.
\smallskip

The implication $(4)\Ra(5)$ is trivial. To prove that $(5)\Ra(4)$, assume that the set $X^\sharp$ is infinite and hence contains a sequence $(p_n)_{n\in\w}$ of pairwise distinct ultrafilters. Since $X^\sharp$ is a closed subspace of the compact Hausdorff space $\beta X$, we can replace $(p_n)_{n\in\w}$ by a suitable subsequence and assume that the subspace $\{p_n\}_{n\in\w}$ is discrete in $X^{\sharp}$.  Fix any free ultrafilter $\mathcal U$ on the set $\w$ and consider the ultrafilter $p=\lim_{n\to\mathcal U}p_n$, consisting of the sets $\bigcup_{n\in U}P_n$ where $U\in\U$ and $P_n\in p_n$ for all $n\in U$. Since the set $X^\sharp$ is closed in $\beta X$, the ultrafilter $p$ belongs to the set $X^\sharp$, being an accumulation point of the set $\{p_n:n\in\w\}$ in $\beta X$. We claim that the ultrafilter $p$ is not countably complete. Using the discreteness of the subspace $\{p_n\}_{n\in\w}$ in $X^\sharp$, in each ultrafilter $p_n$ we can choose a set $P_n$ so that the family $(P_n)_{n\in\w}$ is disjoint. Now observe that for every $n\in\w$ the set $Q_n=\bigcup_{m\ge n}P_m$ belongs to the ultrafilter $p$ but $\bigcap_{n\in\w}Q_n=\emptyset$, witnessing that $p$ is not countably complete. 
\end{proof}



A discrete ballean $X$ is called {\em ultradiscrete} if the family $\{X\setminus B:B\in\mathcal B_X\}$ is an ultrafilter on $X$. Proposition~\ref{p3.5} implies the following characterization.

\begin{corollary}\label{c3.6} For an ultradiscrete ballean $X$,  the  following  statements are  equivalent
\begin{enumerate}
\item $X$  is emu-bounded;
\item $X$  is bo-bounded;
\item  $X$  is so-bounded;
\item the bornology $\mathcal B_X$ of $X$ is $\sigma$-additive.
\end{enumerate}
\end{corollary}

\section{Products}

In this section we discuss the problem of preservation of various notions of boundedness by products of balleans.

\begin{theorem} The product $X\times  Y$  of two mu-bounded  spaces $X, Y$  is mu-bounded.
\end{theorem}

\begin{proof} Let $f: X \times Y\to  \mathbb{R}$  be a mu-function.
For each $x\in X$, the function $f _{x}  :  Y\to  \mathbb{R}$,  defined as $f _{x} (y)= f (x,y)$,  is a mu-function,  so $f_{x}$  is bounded.
It is easy to check that the function $s: X\to  \mathbb{R}$, $s:x\mapsto  \sup  \{|f _{x} (y)|:  y\in Y\}$ is macro-uniform and hence bounded (as $X$ is mu-bounded).
The the function $f$ is bounded as well.
\end{proof}

\begin{theorem}\label{t4.2} Let $X,Y$  be balleans and $Y$ is either bounded  or locally finite. If $X, Y$  are either  emu-bounded or bo-bounded or so-bounded, then  the product  $X\times Y$  has the same property.
\end{theorem}

\begin{proof}  We prove only the third statement.  Assume that the balleans  $X,Y$   are so-bounded. Take any  so-function $f: X \times Y\to  \mathbb{R}$. Assume that  $Y$  is bounded, fix  any $y_{0} \in Y$ and consider the so-function $f_{y_0}:X\to\IR$, $f_{y_0}:x\mapsto f(x,y_0)$. Since $X$ is so-bounded, there exists a bounded set $B\subset X$  such that $f_{y_0}$ is bounded on $X\setminus B$.
Since $Y$ is bounded and $f$ is slowly oscillating, we can replace $B$ by a larger bounded set, and assume that the number $c:=\sup_{x\in X\setminus B}\diam f(\{x\}\times Y)$ is finite. Then $f$  is bounded on $(X\times  Y) \setminus (B \times Y)$.
\smallskip

Now assume that $Y$  is locally finite. In this case, for every $x\in X$ the so-function $f_{x}:  Y \to\mathbb{R}$, $f_x:y\mapsto f(x,y)$, is  bounded.
 Then the function $s:X\to\IR$, $s:x\mapsto \sup_{y\in Y}|f_{x}(y)|$,  is  well-defined and slowly oscillating. Since $X$ is so-bounded, there is $B\in \mathcal{B}_{X}$  such $s$ is  bounded on $X\setminus  B$.
Then $f$ is bounded on  $(X\times Y ) \setminus  (B\times  Y)$.
By first case, there   is a bounded subset $C$ of $Y$ such $f$  is bounded on  $(B\times Y)\setminus  (B\times C)$.
\end{proof}


\begin{question}\label{q4.3}  Is the product of any two emu-bounded balleans  emu-bounded?
\end{question}

\begin{question}  Is the product of any two
    $bo$-bounded balleans   $bo$-bounded?
\end{question}

\begin{question}\label{q4.5}  Is the product of any two    $so$-bounded balleans    $so$-bounded?
\end{question}

For more results on products of balleans, see \cite{BPnorm}.
Now we give some partial answers to Question~\ref{q4.3}--\ref{q4.5}.

For any bornology $\mathcal B$ on a set $X\notin\mathcal B$ consider the cardinal characteristics:
$$
\begin{aligned}
\add(\mathcal B)&=\min\{|\mathcal A|:\mathcal A\subset\mathcal B,\;\;\textstyle{\bigcup\mathcal A\notin\mathcal B}\},\\
\non(\mathcal B)&=\min\{|A|:A\subset X,\;\;A\notin\mathcal B\},\\
\cov(\mathcal B)&=\min\{|\mathcal A|:\mathcal A\subset\mathcal B,\;\;\textstyle{\bigcup\mathcal A=X}\},\\
\cof(\mathcal B)&=\min\{|\mathcal A|:\mathcal A\subset\mathcal B,\;\;\forall B\in\mathcal B\;\exists A\in\mathcal A\;\;(B\subset A)\}
\end{aligned}
$$
It is known (and easy to see) that $$\add(\mathcal B)\le\min\{\cov(\mathcal B),\non(\mathcal B)\}\le\max\{\cov(\mathcal B),\non(\mathcal B)\}\le\cof(\mathcal B).$$

\begin{proposition}\label{p4.6} Let  $X,Y$ be two balleans with $\cof(\mathcal B_X)<\non(\mathcal B_Y)$. If $X$ and $Y$ are emu-bounded (resp. bo-bounded or so-bounded), then so is the product $X \times Y$.
\end{proposition}

\begin{proof} We shall present a proof only for the case of emu-boundedness. The other two cases can be considered by analogy. Assume that the balleans $X,Y$ are emu-bounded. If the ballean $X$ is bounded, then the product $X\times Y$ is emu-bounded by Theorem~\ref{t4.2}. So we assume that $X$ is unbounded. By the definition of the cardinal $\kappa=\cof(\mathcal B_X)$, the bornology $\mathcal B_X$ of $X$ has a base $\mathcal B\subset\mathcal B_X$ of cardinality $|\mathcal B|=\kappa$.

To prove that the product $X\times Y$ is emu-bounded, take any emu-function $f:X\times Y\to\IR$. Fix any point $y_0\in Y$. Since the function $f_{y_0}:X\to\IR$, $f_{y_0}:x\mapsto f(x,y_0)$, is emu-bounded, for any entourage  $E\in\mathcal{E}_{Y}$ there exists a bounded set $B_E\in\mathcal B$ such  that
$$c_E:=\sup_{x\in X\setminus B_E}\diam \  f(\{x\}\times E [y_{0}])$$ is finite. We claim that for some $B\in\mathcal B$ and $m\in\mathbb N$ the union $$U_{B,m}:={\textstyle\bigcup}\{E[y_0]:E\in\mathcal E_Y,\;B_E=B,\;c_E\le m\}$$ equals $Y$. To derive a contradiction, assume that for every $B\in \mathcal B$ and $m\in\mathbb N$ the union $U_{B,m}$ does not contain some point $y_{B,m}\in Y$. Since $\w\le\kappa<\non(\mathcal B_Y)$, the set $V=\{y_{B,m}:B\in\mathcal B,\;m\in\mathbb N\}$ is bounded in $Y$ and hence is contained in the ball $E[y_0]$ for some $E\in\mathcal E_Y$. Then for the set $B=B_E$ and any number $m\in\mathbb N$ with $m\ge c_E$, we get $y_{B,m}\in E[y_0]\subset U_{B,m}$, which contradicts the choice of $y_{B,m}$. This contradiction shows that $U_{B,m}=Y$ for some $B\in\mathcal B$ and $m\in\mathbb N$.

Then $\diam f(\{x\}\times Y)\le 2m$ for each $x\in X\setminus B$.
Since $X$ is emu-bounded, there exists $D\in\mathcal{B}_{X}$
 such that the number $c:=\sup_{x\in X\setminus D}|f(x,y_0)|$ is finite. Replacing $D$ by $B\cup D$, we can assume that $B\subset D$.
Then  $|f(x, y)|< c+2m$
for all
$(x,y)\in (X\setminus D)\times Y$.
By Theorem~\ref{t4.2}, the ballean $D\times Y$ is  emu-bounded, which implies that the function $f$ is bounded on the complement $D\times Y\setminus D\times C$ for some bounded set $C\in\mathcal B_Y$. Now we see that $f$ is bounded on $X\times Y\setminus D\times C$, witnessing that the ballean $X\times Y$ is emu-bounded.
\end{proof}

\begin{proposition} Let  $X,Y$ be two so-bounded balleans. If $\cov(\mathcal B_X)<\add(\mathcal B_Y)$, then the product $X \times Y$ is so-bounded.
\end{proposition}

\begin{proof} 
 By the definition of the cardinal $\kappa=\cov(\mathcal B_X)$, there exists a subfamily $\{B_\alpha\}_{\alpha\in\kappa}\subset\mathcal B_X$ with $\bigcup_{\alpha\in\kappa}B_\alpha\in X$.

To show that the ballean $X\times Y$ is so-bounded, take any so-function $f:X\times Y\to\IR$. Fix any point $x_0\in X$ and for every $\alpha\in \kappa$ find an entourage $E_\alpha$ on $X$ such that $B_\alpha\subset E_\alpha[x_0]$.
Then $X=\bigcup_{\alpha\in \kappa}E_\alpha[x_0]$. 
 Since  $f$ is slowly oscillating,  there  exists
$C_{\alpha}\in   \mathcal{B}_{Y} $
  such  that  $\sup_{y\in Y\setminus C_\alpha}\diam \  f(E_{\alpha} [x_{0}]\times\{y\})\le 1.$  Since $\add(\mathcal B_Y)>\cov(\mathcal B_X)=\kappa$, the set $C= \bigcup_{\alpha<\kappa} C_{\alpha}$  is bounded  in $Y$.
If $y\in  Y\setminus  C$  then $\diam f(E_\alpha[x_0]\times\{y\})\le 1$ for every $\alpha\in\kappa$ and hence $\diam f(X\times\{y\})\le 2$.

Since $Y$ is so-bounded,  there exist a bounded set
 $D\in \mathcal{B}_{Y}$   such that $\sup_{y\in Y\setminus D}|f(x_0,y)|<1$. Replacing $D$ by $C\cup D$, we can assume that $C\subset D$. Then $|f(x,y)|<1+\diam f(X\times \{y\})\le 1+2=3$ for all
 $x\in X$ and  $y\in Y \setminus D$.
 Therefore, $f$ is bounded on the set $(X\times Y)\setminus (X\times D)$. 
 
By Theorem~\ref{t4.2}, $X\times D$
  is so-bounded so there exists a bounded set $B\subset X$ such that $f$ is bounded on the set $(X\times D)\setminus (B\times D)$. Summing up, we conclude that $f$ is bounded on the complement of the bounded set $B\times D$.
  \end{proof}

\begin{proposition} Let  $X,Y$ be two balleans such that $\cov(\mathcal B_X)<\add(\mathcal B_Y)$ and the bornology $\mathcal B_X$ is tall. If $X$ and $Y$ are emu-bounded (resp. bo-bounded), then so is the product $X \times Y$.
\end{proposition}

\begin{proof} We shall provide a detail proof only for the case of emu-boundedness. The case of bo-boundedness can be considered by analogy. So assume that the balleans $X,Y$ are emu-bounded.
 By the definition of the cardinal $\kappa=\cov(\mathcal B_X)$, there exists a subfamily $\{B_\alpha\}_{\alpha\in\kappa}\subset\mathcal B_X$ with $\bigcup_{\alpha\in\kappa}B_\alpha\in X$.

To show that the ballean $X\times Y$ is emu-bounded, take any emu-function $f:X\times Y\to\IR$. Fix any point $x_0\in X$ and for every $\alpha\in \kappa$ find an entourage $E_\alpha$ on $X$ such that $B_\alpha\subset E_\alpha[x_0]$.
 Since  $f$ is eventually macro-uniform,  there  exists
$C_{\alpha}\in   \mathcal{B}_{Y} $
  such  that  $$c_\alpha:=\sup_{y\in Y\setminus C_\alpha}\diam \  f(E_{\alpha} [x_{0}]\times\{y\})$$is finite. For every $m\in\IN$ let $K_m=\{\alpha\in\kappa:c_\alpha\le m\}$. We claim that for some $m\in\IN$ the union $U_m=\bigcup_{\alpha\in K_m}E_\alpha[x_0]$ coincides with $X$. In the opposite case, for every $m\in\IN$ we can choose a point $x_m\in X\setminus U_m$. Since the bornology $\mathcal B_X$ is tall, there exists an increasing number sequence $(m_k)_{k\in\w}$ such that the set $\{x_{m_k}\}_{k\in\w}$ is bounded and hence is contained in $B_\alpha\subset E_\alpha[x_0]$ for some $\alpha\in\kappa$. Then for any $k\in\w$ with $m_k\ge c_\alpha$ we have $x_{m_k}\in E_\alpha[x_0]\subset U_{m_k}$, which contradicts the choice of $x_{m_k}$. This contradiction shows that $U_m=X$ for some $m\in\IN$.

Since $\add(\mathcal B_Y)>\cov(\mathcal B_X)=\kappa\ge|K_m|$, the set $C= \bigcup_{\alpha\in K_m} C_{\alpha}$  is bounded  in $Y$.
If $y\in  Y\setminus  C$  then $\diam f(E_\alpha[x_0]\times\{y\})\le c_\alpha\le m$ for every $\alpha\in K_m$ and hence $\diam f(X\times\{y\})=\diam(U_m\times\{y\})\le 2m$.

Since $Y$ is emu-bounded,  there exist a bounded set
  $D\in \mathcal{B}_{Y}$  and  such that the number 
  $$c:=\sup_{y\in Y\setminus D}|f(x_0,y)|$$ is finite. Replacing the set $D$ by $C\cup D$, we can assume that $C\subset D$. Then $|f(x,y)|<c+\diam f(X\times \{y\})<c+2m$ for all
 $x\in X$ and  $y\in Y \setminus D$. 
 Therefore, $f$ is bounded on the set $(X\times Y)\setminus (X\times D)$. 
 
By Theorem~\ref{t4.2}, the ballean $X\times D$
  is emu-bounded, so there exists a bounded set $B\subset X$ such that $f$ is bounded on the set $(X\times D)\setminus (B\times D)$. Summing up, we conclude that $f$ is bounded on $(X\times Y)\setminus(B\times D)$, witnessing that $X\times Y$ is emu-bounded.
  \end{proof}

\section{Group balleans}

A bornology $\mathcal B$ on a group $G$ is called a {\em group bornology} if for any sets $A,B\in\mathcal B$ the set $AB^{-1}=\{ab^{-1}:a\in A,\;b\in B\}$ belongs to $\mathcal B$. 

Each group bornology $\mathcal B$ on a group $G$ induces the coarse structure $\mathcal E_{\mathcal B}$, generated by the base $\{E_B:B\in\mathcal B\}$ consisting of the entourages $E_B=\{(x,y)\in G\times G:y\in \{x\}\cup Bx\}$ where $B\in\mathcal B$.

It is easy to see that for any infinite cardinal $\kappa$ of any group $G$ the family $$[G]^{<\kappa}:=\{B\subset G:|B|<\kappa\}$$ is a group bornology on $G$. It induces the coarse structure $\mathcal E_{[G]^{<\kappa}}$ on $G$. The ballean $(G,\mathcal E_{[G]^{<\kappa}})$ will be called the {\em $\kappa$-ballean} of $G$. The $\w$-ballean of $G$ is called the {\em finitary ballean} of $G$.  It should be mentioned that (in metic terms) the finitary balleans of finitely generated  groups $G$ have been intensively studied in
 Geometrical Group Theory,  see \cite[Chapter 4]{b17}. For properties of the $\kappa$-balleans of groups, see \cite{b21}, \cite{b18}.

\begin{theorem}\label{t5.1} For any infinite regular cardinal $\kappa$, the $\kappa$-ballean of any group $G$ of cardinality $|G|=\kappa$ is not so-bounded.
\end{theorem}

\begin{proof} Since $\kappa$ is regular, the coarse structure $\mathcal{E} _{[G]^{<\kappa}}$ has a linearly ordered base. Applying Proposition 3.1 of \cite{b19}, we conclude that the $\kappa$-ballean $(G,\mathcal E_{[G]^{<\kappa}})$ is not so-bounded.
\end{proof}

Example  3.3  from \cite{FP}  shows  that the finitary ballean of any infinite free group $F(X)$ is  not so-bounded.

\begin{theorem}\label{t5.2} For any infinite cardinal $\kappa$ and set $X$ of cardinality $|X|\ge\kappa$ the $\kappa$-ballean of the free group $G=F(X)$ is not so-bounded.
\end{theorem}

\begin{proof}
We partition $X=\bigcup  _{n\in \mathbb{N} } X_{n} $ so that  $|X_{n}|= |X|$  for each $n\in \mathbb{N}$,  and  define  a  function $h: X\cup \{e\}\cup  X^{-1}\to\mathbb{R}$ by $h(e)=0$ and $h(x)= h(x^{-1})=n$ for any  $x\in X_{n}$, $n\in\IN$.
Then we define a mapping  $f: G \to\mathbb{R}$  assigning to each word $w\in G=F(X)$ the number $h(a)$ where $a$  is the last letter in the irreducible representation of $w$.
It is easy to see that $f$ is unbounded  the complement $G\setminus B$ of any set $B\in[G]^{<\kappa}$.

To show that $f$ is slowly oscillating,  we take an arbitrary subset  $B\in [G]^{<\kappa }$  and denote by  $S$  the subgroup  of  $G$  generated by all letters which
appear in the  (irreducible) words from  $B$.
If $w\in  G\setminus  S$  and  $y\in S$  then  $f(y w)= f(w)$,  so $\diam     f(Sw)=0$.  
\end{proof}

Following \cite{b20}, we say that a group $G$  is $\kappa$-{\it normal}  if every   subset $F\in [G] ^{<\kappa}$  is contained in a normal subgroup  $H\in[G] ^{<\kappa}$ of $G$.

\begin{theorem}\label{t5.3} Let $\kappa$ be a cardinal of uncountable cofinality and $G$ be a  group of cardinality $|G|>\kappa$. If the group $G$ is $\kappa^+$-normal, then the $\kappa$-ballean of $G$ is emu-bounded.
\end{theorem}

\begin{proof} Given any emu-function $f:G\to\IR$, we should prove that $f$ is bounded on the complement $G\setminus B$ of some set $B\in[G]^{<\kappa}$. First, we prove two claims.

\begin{claim}\label{cl1} There  exist a normal subgroup  $H\subset G$  of cardinality $\kappa$ such that $$\sup_{x\in G\setminus H}\diam f(Hx)<\infty.$$
\end{claim}

\begin{proof} Since $G$ is a $\kappa^+$-normal group of cardinality $|G|>\kappa$, there exists a normal subgroup $H_0\subset G$ of cardinality $|H_0|=\kappa$. By transfinite induction of length $\w_1$ we shall construct an increasing sequence $(H_\alpha)_{\alpha\in\w_1}$ of normal subgroups of cardinality $\kappa$ such that for any ordinal $\alpha<\w_1$ the supremum $\sup_{x\in X\setminus H_{\alpha+1}}\diam f(H_\alpha x)$ is finite.

Assume that for some non-zero ordinal $\alpha<\w_1$ we have constructed an increasing sequence $(H_\beta)_{\beta<\alpha}$ of normal subgroups of cardinality $\kappa$ in $G$ such that for any ordinal $\beta$ with $\beta+1<\alpha$ the supremum $\sup_{x\in X\setminus H_{\beta+1}}\diam f(H_\beta x)$ is finite. If $\alpha$ is a limit ordinal, put $H_\alpha=\bigcup_{\beta<\alpha}H_\beta$.

If $\alpha=\beta+1$ is a successor ordinal, then write the group $H_\beta$ as the union $H_\beta=\bigcup_{i\in\cf(\kappa)}\Gamma_i$ of an increasing sequence $(\Gamma_i)_{i\in\cf(\kappa)}$ of sets of cardinality $|\Gamma_i|<\kappa$. Since $f$ is eventually macro-uniform, for every $i\in\cf(\kappa)$ there exists a set $B_{i}\in [G]^{<\kappa}$ such that 
$c_i:=\sup_{x\in G\setminus B_{i}}\diam f(\Gamma_i x)$ is finite. Since $G$ is $\kappa^+$-normal, the set $H_\beta\cup\bigcup_{i\in\cf(\kappa)}B_i$ is contained in a normal subgroup $H_\alpha$ of cardinality $|H_\alpha|=\kappa$. Since the cardinal $\cf(\kappa)$ is uncountable and regular, there exists $c\in\IN$ such that the set $\Omega=\{i\in\cf(\kappa): c_i\le c\}$ is unbounded in $\cf(\kappa)$. Then $H_\beta=\bigcup_{i\in\Omega}\Gamma_i$ and for any $x\in X\setminus H_\alpha=X\setminus H_{\beta+1}$ 
$$\diam f(H_\beta x)=\sup_{i\in\Omega}\diam f(\Gamma_ix)\le\sup_{i\in\Omega}c_i\le c.$$
\end{proof}

\begin{claim}\label{cl2} The function $f$ is bounded on the set $G\setminus H$.
\end{claim}

\begin{proof} To derive a contradiction, assume that $f$  is  unbounded on  $G\setminus  H$  and  choose a sequence  $(g_{n}) _{n\in \omega}$  in  $G\setminus  H$  such that  the set $\{f(g_{n}) : n\in \omega\}$
 is unbounded. We put 
 $K=\{  g_{n} g_{0} ^{-1}  : n\in\w\}$  and  choose
 $B\in  [G] ^{< \kappa}$  such that the number $$C:=\sup_{x\in X\setminus B}\diam f(Kx)$$ is finite. By Claim~\ref{cl1}, the supremum $$c:=\sup_{x\in G\setminus H}\diam f(Hx)$$is finite. Since $|H|=\kappa>|B|=|g_0^{-1}B|$,  there exits $x\in H\setminus g_0^{-1}B$.
Then $g_0x\notin B$ and 
$|f (g_{n} g_{0} ^{-1} (g_{0} x )) -  f(g_{0} x ) | \le\diam f(Kx)\le C $
 for each  $n>0$.
Since $H$  is normal,
 $g_{n} x \in g_nH=Hg_{n}$  and then
$$
\begin{aligned}
|f(g_n)|&=|f(g_n)-f(g_nx)+f(g_nx)-f(g_0x)+f(g_0x)|\le\\
&\le |f(g_n)-f(g_nx)|+|f(g_nx)-f(g_0x)|+|f(g_0x)|<\\
&<\mathrm{diam}\,f(Hg_n)+C+|f(g_0x)|<c+C+|f(g_0x)|
\end{aligned}$$for every $n$, which contradicts the choice of  $(g_{n}) _{n\in \omega}$.
This contradiction shows that $f$   is bounded on  $G \setminus H$.
\end{proof}

Fix any element $g\in G\setminus H$ and observe that $gH\subset G\setminus H$ and hence the set $f(gH)\subset f(G\setminus H)$ is bounded in the real line (according to Claim~\ref{cl2}).
Consequently, the supremum $$c:=\sup_{x\in G\setminus H}|f(x)|$$ is finite.
  
Since $f$ is eventually macro-uniform, there exists a set $B\in [G] ^{<\kappa}$  such that the number $$c':=\sup_{x\in G\setminus B}|f (gx)  -  f(x) |$$ is finite. Now take any $x\in H\setminus B$ and observe that 
$|f(x)|\le |f(gx)-f(x)|+|f(gx)|\le c'+c$, witnessing that $f$ is bounded on $H\setminus B$ and consequently, bounded on $G\setminus B=(G\setminus H)\cup(H\setminus B)$.
\end{proof}

A subset  $A$ of a group $G$ is defined to be $\kappa$-{\it large}  if  $G=FA$ for some $F\in [G] ^{<\kappa}$. In the proof of Theorem~\ref{t5.4} we shall need the following result, proved in \cite{b20}.
 
\begin{theorem}[Protasov, Slobodianiuk]\label{t:BS} Let $\kappa$ be a singular  cardinal and let $G$ be a  $\kappa$-normal group of  cardinality $\kappa$.  For  every finite partition of $G$,  at  least one cell of the partition is $\kappa$-large.
\end{theorem}

\begin{theorem}\label{t5.4} Let $\kappa$ be a singular  cardinal and let $G$ be a 
group of cardinality $|G|=\kappa$. If the group $G$ is $\kappa$-normal, then the $\kappa$-ballean of $G$ is emu-bounded.
\end{theorem}

\begin{proof}
Assuming that the $\kappa$-ballean of $G$ is not emu-bounded, we can find an emu-function  $f: G\to \IN$ such that $f(G\setminus B)$ is unbounded for every $|B|\in[G]^{<\kappa}$. This property of $f$ implies that for any $m\in\IN$ the preimage $f^{-1}([m,\infty))$ has cardinality $\kappa$.

\begin{claim} There exists a number sequence $(n_i)_{i\in\w}$ such that $n_0=0$,  $n_{i+1}-n_i>i$ for all $i\in\w$ and for every $m\in\w$ and $r\in\{0,1,2,3\}$ the set $$X_{m,r}:=\bigcup_{k=m}^\infty f^{-1}([n_{4k+r},n_{4k+r+1}))$$ has cardinality $|X_{m,r}|=\kappa$.
\end{claim}

\begin{proof} Separately we shall consider two cases. 
\smallskip

1. The set $W=\{n\in\w:|f^{-1}([n,n+1))|=\kappa\}$ is infinite. In this case let $\{n_i\}_{i\in\w}\subset\w$ be any increasing number  sequence such that $n_0=0$, $n_{i+1}-n_i>i$ and $[n_{i+1},n_i)\cap W\ne\emptyset$ for all $i\in\w$. It is clear that the sequence $(n_i)_{i\in\w}$ has the required property.
\smallskip

2. The set $W$ is finite. Put $n_0=0$ and find a number $n_1\in\IN$ with $W\subset [n_0,n_1)$. The definition of $W$ implies that for any $m>n_1$ the preimage $f^{-1}([n_1,m))$ has cardinality $<\kappa$. Taking into account that $f^{-1}([n_1,\infty))=\bigcup_{m>n_1}f^{-1}([n_1,m))$ has cardinality $\kappa$, we conclude that the cardinal $\kappa$ has countable cofinality. So, we can find a strictly increasing sequence of cardinals $(\kappa_i)_{i\in\w}$ with $\sup_{i\in\w}\kappa_i=\kappa$. Since $f^{-1}([n_1,\infty))=\bigcup_{n>m}f^{-1}([n_1,m))$ has cardinality $\kappa$, there exists a number $n_2>n_1+1$ such that $|f^{-1}([n_1,n_2))|\ge \kappa_1$. Proceeding by induction, for every $i\in\IN$ we can find a number $n_{i+1}>n_i+i$ such that $|f^{-1}([n_i,n_{i+1}))|>\kappa_i$. It is easy to see that the sequence $(n_i)_{i\in\w}$ has the required property.
\end{proof}

Since $G = \bigcup_{r=0}^3X_{0,r}$, we can apply Theorem~\ref{t:BS} and find a number $r\in\{0,1,2,3\}$ such that the set $X_{0,r}$ is $\kappa$-large in $G$ and hence  $G=FX_{0,r}$ for some set $F\in [G]^{<\kappa}$ containing the unit $e$ of $G$.

Since the function $f$ is eventually macro-uniform, there exist a set
$B\in [G]^{<\kappa} $ such that $\sup_{x\in G\setminus B}\diam f(Fx)<m$ for some number $m\in\IN$. Let $q\in\{0,1,2,3\}$ be a unique number with $|q-r|=2$. Since $|X_{m,q}|=\kappa>|FB|$, there exists an element $y\in X_{m,q}\setminus FB$. Taking into account that $y\in G\setminus FB=FX_{0,r}\setminus FB=F(X_{0,r}\setminus B)$, we see that $y\in Fx$ for some $x\in X_{0,r}\setminus B$. The choice of $B$ ensures that $\diam f(Fx)\le m$ and hence $|f(y)-f(x)|\le \diam f(Fx)\le m$. It follows from $x\in X_{0,r}=\bigcup_{k=0}^\infty f^{-1}([n_{4k+r},n_{4k+r+1}))$ that $f(x)\in [n_{4i+r},n_{4i+r+1})$ for some $i\in\w$. On the other hand, $y\in X_{m,q}=\bigcup_{k=m}^\infty f^{-1}([n_{4k+q},n_{4k+q+1}))$ yields a number $j\ge m$ such that $f(y)\in [n_{4j+q},n_{4j+q+1})$. Now we arrive to a contradiction with $|f(y)-f(x)|\le m$, considering the following four cases.

If $i<j$, then $4i+r+1\le 4(j-1)+1+r\le 4j-3+q+|r-q|=4j+q-1$ and hence 
$$f(y)\ge n_{4j+q}>n_{4j+q-1}+(4j+q-1)\ge n_{4i+r+1}+(4m+q-1)>f(x)+m.$$

If $j<i$, then $4j+q+1\le 4(i-1)+1+q\le 4i-3+r+|r-q|=4i+r-1$ and hence 
$$f(x)\ge n_{4i+r}>n_{4i+r-1}+(4i+r-1)\ge n_{4j+q+1}+(4j+r+3)>f(y)+m.$$

If $i=j$ and $r<q$, then 
$$f(y)\ge n_{4j+q}>n_{4j+q-1}+(4j+q-1)=n_{4j+r+1}+(4j+q-1)>f(x)+m.$$

If $i=j$ and $q<r$, then 
$$f(x)\ge n_{4i+r}>n_{4i+r-1}+(4i+r-1)=n_{4j+q+1}+(4j+r-1)>f(y)+m.$$
\end{proof}

\begin{proposition}\label{t5.5} If  the  $\kappa$-ballean of each subgroup $H$ of $G$ with  $|H|= \kappa$  is emu-bounded,  then the  $\kappa$-ballean of $G$  is  emu-bounded.
\end{proposition}

\begin{proof} Assuming that the $\kappa$-ballean of $G$ is not emu-bounded, we can find an emu-function  $f: G \to \mathbb{N}$, which is unbounded  on each subset  $G \setminus B$,  $|B|< \kappa$. This property of $f$ implies that for every $m\in\IN$ the set $f^{-1}([m,\infty)$ has cardinality $\ge\kappa$ and hence contains a subset $X_m$ of cardinality $|X_m|=\kappa$. Then the set $\bigcup_{m\in\w}X_m$ generates a subgroup $H\subset G$ of cardinality $|H|=\kappa$ such that $f(H\setminus B)$ is unbounded for any $B\in[H]^{<\kappa}$. Now we see that the restriction $f{\restriction}H$ witnesses that $H$ is not emu-bounded.
\end{proof}

\begin{remark} Theorem~\ref{t5.5}  shows that  Theorem~\ref{t5.4}  is true if $|G|\geq\kappa $.
By Theorem~\ref{t5.1}, \ref{t5.2}, \ref{t5.3}, \ref{t5.4}, \ref{t5.5},
for an uncountable cardinal  $\kappa$,
 the   $\kappa$-ballean of an  Abelian group  $G$ is emu-bounded if and only if
  either $\kappa$ is singular or  $\kappa$  is regular and $|G|> \kappa$.
In this case, the  emu-boundedness is equivalent to the so-boundedness.
If $\kappa$ is regular  and $G$  is an Abelian group of cardinality $|G|>\kappa$,  then the
$\kappa$-ballean of $G$ is emu-bounded but the  $\kappa$-ballean of every subgroup
$H$,  $|H|=\kappa $ is not emu-bounded.
\end{remark}

\begin{remark} By Theorems~\ref{t5.1} and \ref{t5.3}, the $\w_{1}$-ballean of the group  $\mathbb{R}$ is emu-bounded  if and only if CH holds.  On the other  hand,  by  Theorem~\ref{t5.2}, the $\w_1$-ballean of the free group  $F({\mathbb{R}})$  is not emu-bounded in ZFC.
\end{remark}

\begin{question} Is the $\w_1$-ballean of the permutation group  $S_{\omega}$ emu-bounded  in ZFC?
\end{question}


By analogy with Theorem 3.1 of \cite{FP}, the following theorem can be proved.

\begin{theorem}\label{t4.5} If $G$ is an uncountable $\w_1$-normal group, then
 the finitary ballean of $G$  is  bo-bounded and so-bounded,  and  every so-function  on $G$  is constant at infinity.
\end{theorem}

\begin{question} Assume that the finitary ballean of a group $G$ is so-bounded.
Is every so-function on $G$ constant at infinity?
\end{question}

\begin{question} Is the so-boundedness of the finitary ballean of a group equivalent to its bo-boundedness?
\end{question}

Finally, we shall discuss the mu-boundedness of finitary balleans on groups.
Since the fintary ballean of any group is locally finite, each emu-function is a mu-function. Consequently, the finitary ballean of any group is mu-bounded if and only if it is emu-bounded.

\begin{definition}
A group $G$ is defined to be
\begin{itemize}
\item {\em $n$-Shelah} for some $n\in\IN$ if for each subset $A\subset G$ of cardinality $|A|=|G|$ we have $A^n=G$;
\item {\em Shelah} if it is $n$-Shelah for some $n\in\IN$;
\item {\em almost Shelah}  if for each subset $A\subset G$ of cardinality $|A|=|G|$ there exists $n\in\IN$ such that $A^n=G$;
\item {\em J\'onsson} if each subsemigroup $A\subset G$ of cofinality $|A|=|G|$ coincides with $G$;
\item {\em Kurosh} if each subgroup $A\subset G$ of cofinality $|A|=|G|$ coincides with $G$;
\item {\em Bergman} if $G$ cannot be written as the union a strictly increasing sequence $(X_n)_{n\in\w}$ of subsets such that $X_n=X_n^{-1}$ and $X_n^2\subset X_{n+1}$ for all $n\in\w$.
\end{itemize}
\end{definition}
For any group $G$ the following implications hold:
$$\xymatrix{
\mbox{finite}\ar@{<=>}[r]&\mbox{1-Shelah}\ar@{=>}[r]&\mbox{Shelah}\ar@{=>}[r]&\mbox{almost Shelah}\ar@{=>}[r]&\mbox{J\'onsson}\ar@{=>}[r]&
\mbox{Kurosh}
}
$$
In \cite{Shelah} Shelah constructed a ZFC-example of an uncountable J\'onsson group and a CH-example of an infinite 6640-Shelah group. More precisely, for every cardinal $\kappa$ with $\kappa^+=2^\kappa$ there exists a 6640-Shelah group of cardinality $\kappa^+=2^\kappa$. In \cite{Berg} Bergman proved that the permutation group $S_X$ of any set $X$ is Bergman. Later it was shown \cite{DT}, \cite{Tol} that many automorphism groups (of sufficiently homogeneous structures) are Bergman.


\begin{theorem} \label{l} For a  group $G$ the following conditions are equivalent:
\begin{enumerate}
\item The finitary ballean of $G$ is mu-bounded.
\item The finitary ballean of $G$ is emu-bounded.
\item The group $G$ is Bergman.
\end{enumerate}
The conditions \textup{(1)--(3)} follow from
\begin{enumerate}
\item[(4)] $G$ is almost Shelah and $\cf(|G|)>\w$.
\end{enumerate}
\end{theorem}

\begin{proof} The equivalence $(1)\Leftrightarrow(2)$ follows from the local finiteness of the finitary ballean on $G$.

$(1)\Ra(3)$ Assume that the group $G$ fails to be Bergman. Then $G=\bigcup_{n\in\w}X_n$ for some strictly increasing sequence of subsets of $G$ such that $X_n=X_n^{-1}$ and $X_n^2\subset X_{n+1}$ for all $n\in\w$. We lose no generality assuming that $X_0=\emptyset$.

We claim that the function $\varphi:G\to\IN\subset\IR$ assigning to each $x\in X$ the unique number $n\in\IN$ such that $x\in X_{n}\setminus X_{n-1}$ is macro-uniform.
It follows from $X_n^2\subset X_{n+1}$ that $$\varphi(gx)\le1+\max\{\varphi(g),\varphi(x)\}\le \varphi(x)+(1+\varphi(g))$$ for every $g,x$. Then $\varphi(x)=\varphi(g^{-1}gx)\le \varphi(gx)+(1+\varphi(g^{-1}))$ and hence $\varphi(gx)\ge \varphi(x)-(1+\varphi(g^{-1})$.
Therefore,
$$\varphi(x)-(1+\varphi(g^{-1}))\le \varphi(gx)\le\varphi(x)+(1+\varphi(g)),$$
which implies that the map $\varphi$ is macro-uniform. Since the sequence $(X_n)_{n\in\w}$ is strictly increasing, the map $\varphi$ is unbounded, so the finitary ballean of $G$ is not $mu$-bounded.
\smallskip

$(3)\Ra(1)$  Assume that $G$ admits an unbounded macro-uniform function $\varphi:G\to\IR$. We lose no generality assuming that $\varphi(e)=0$ where $e$ is the unit of the group $G$. Then for every finite set $F\subset G$ there exists a number $n_F\in\IN$ such that $\varphi(Fx)\subset [-n_F,n_F]+\varphi(x)$ for every $x\in G$.
For every $n\in\IN$ let $X_n=\{g\in G:n_{\{g,g^{-1}\}}\le n\}$ and observe that $(X_n)_{n\in\w}$ is an increasing sequence such that $X_n=X_n^{-1}$ and $\bigcup_{n\in\w}X_n=G$. For every $n\in\w$ let $Y_n=X_n^{2^n}$ and observe that $Y_n=Y_n^{-1}$ and $Y_n\cdot Y_n=X_n^{2^n}X_n^{2^n}=X_n^{2^{n+1}}\subset X_{n+1}^{2^{n+1}}=Y_{n+1}$.

Assuming that $G$ is Bergman, we conclude that $G=Y_n=X_n^{2^n}$ for some $n\in\w$.

By induction we shall prove that $\varphi(X_n^k)\subset[-nk,nk]$ for $k\in\w$. For $k=0$ we have $\varphi(X_n^0)=\varphi(\{e\})=\{0\}=[-n0,n0]$. Assume that for some $k\in\w$ we have proved that $\varphi(X_n^k)\subset[-nk,nk]$. Taking into account that for every $g\in X_n$ and $x\in G$ we have $\varphi(gx)\subset \varphi(x)+[-n_{\{g\}},n_{\{g\}}]\subset \varphi(x)+[-n,n]$, we conclude that $\varphi(X_nx)\subset \varphi(x)+[-n,n]$ and hence $$\varphi(X_n^{k+1})=\varphi(X_n\cdot X_n^k)\subset \varphi(X_n^k)+[-n,n]\subset[-nk,nk]+[-n,n]\subset[-n(k+1),n(k+1)].$$
Then $\varphi(G)=\varphi(X_n^{2^n})\subset[-n2^n,n2^n]$, which means that the function $\varphi$ is bounded.
\smallskip

Finally we prove that $(4)\Ra(3)$. Assume that $G$ is almost Shelah and $\cf(|G|)>\w$.  Assume that $G=\bigcup_{n\in\w}X_n$ for some increasing sequence of subsets such that $X_n=X_n^{-1}$ and $X_n\cdot X_n\subset X_{n+1}$ for all $n\in\w$. By induction we can show that $X_n^{2^k}\subset X_{n+k}$ for every $k\in\w$.

 Since $\cf(|G|)>\w$, there exists $n\in\w$ such that $|X_n|=|G|$. Since $G$ is almost Shelah, $G=X_n^{2^m}\subset X_{n+m}$ for some $m\in\w$ and hence $X_{n+m}=G$, which contradicts the choice of the sequence $(X_k)_{k\in\w}$.
\end{proof}

\begin{corollary} The finitary ballean of the permutation group $S_X$ of any set $X$ is mu-bounded.
\end{corollary}

\begin{problem} Find the largest $n$ for which every $n$-Shelah group is finite.
\end{problem}

\begin{remark} Answering the problem \cite{MO}, Yves Cornulier proved that each 3-Shelah group is finite.
\end{remark}

\vskip 5pt

\begin{question} Is every so-function on a Bergman group (in particuar, on the group  $S_{X}$) constant at infinity?
\end{question}

\end{document}